\documentclass[reqno,12pt]{article}

\usepackage[a4paper, hdivide={2.54cm,,2.54cm},vdivide={3.48cm,,3.48cm}]{geometry}
\usepackage{amssymb}
\usepackage{amsmath}
\usepackage{amsthm}
\usepackage{enumerate}
\usepackage{amstext}
\usepackage{graphicx}			
\usepackage{hhline}
\usepackage{psfrag}
\usepackage{float}
\usepackage[dvips]{epsfig}
\usepackage{rotating}
\usepackage{varioref}
\usepackage{color}
\usepackage{setspace}

\usepackage[utf8]{inputenc}
\usepackage[T1]{fontenc}

\usepackage{fancyhdr}
\allowdisplaybreaks

\DeclareSymbolFont{msbm}{U}{msb}{m}{n}
\DeclareMathSymbol{\N}{\mathalpha}{msbm}{'116}
\DeclareMathSymbol{\R}{\mathalpha}{msbm}{'122}

\def\Om{\Omega}
\def\Omclo{\overline{\Omega}}
\def\pdi{\partial_i}
\def\pdj{\partial_j}

\def\pdn{\partial_n}
\def\pdx{\partial_x}

\def\pdnvtilde{\partial_{\tilde v_n}}

\def\sup{\textnormal{sup}}

\def\dist{\text{dist}}
\def\N{\mathbb N }

\def\exp{\text{exp}}

\def\limn{\underset{n \to \infty}{\text{lim}}}

\def\cutoff{\eta}
\def\ellip{\lambda}
\def\coeffbound{C}
\newcommand{\head}[1]{\hat{#1}}

\def\ton{\, \textnormal{on} \,}
\def\tif{\, \textnormal{if} \,}
\def\telse{\, \textnormal{else} \,}
\def\Id{\mathbf 1}

\def\subtitle{Feynman Formula for smooth bounded domains}

\title{Smooth Contractive Embeddings and Application to
Feynman Formula for Parabolic Equations on Smooth Bounded Domains
}

\author{Benedict Baur, Florian Conrad, Martin Grothaus}

\newtheoremstyle{theorem}%
{3pt}% Space above
{3pt}% Space below
{\itshape}% Body font
{}%
{\bfseries}% Theorem head font
{}% Punctuation after theorem head
{.5em}% Space after theorem head
{}%

\newtheoremstyle{lemma}%
{3pt}% Space above
{3pt}% Space below
{}% Body font
{}%
{\bfseries}% Theorem head font
{}% Punctuation after theorem head
{.5em}% Space after theorem head
{}%

\newtheoremstyle{definition}%
{3pt}% Space above
{3pt}% Space below
{}% Body font
{}%
{\bfseries}% Theorem head font
{}% Punctuation after theorem head
{.5em}% Space after theorem head
{}%

\newtheoremstyle{remark}%
{3pt}% Space above
{3pt}% Space below
{}% Body font
{}%
{\bfseries}% Theorem head font
{:}% Punctuation after theorem head
{.5em}% Space after theorem head
{}%

\theoremstyle{theorem} % Text italic, Heading fat cat

\newtheorem{theorem}{Theorem}[section] % Beginn der section mit theorem, Bezeichnung Theorem, neuzählen nach section

\theoremstyle{lemma} % Text normal, Heading fett
\newtheorem{lemma}[theorem]{Lemma}

\theoremstyle{definition} % Text normal, heading fett
\newtheorem{definition}[theorem]{Definition}
\newtheorem{assumption}[theorem]{Assumption}

\theoremstyle{remark}
\newtheorem{remark}[theorem]{Remark}

\begin{document}
\lhead{\subtitle }{\leftmark}
\rhead{\thepage }{\thepage}

%\maketitle 
\begin{center}
\begin{LARGE}Smooth Contractive Embeddings and Application to
Feynman Formula for Parabolic Equations on Smooth Bounded Domains\end{LARGE} \\
\begin{small}
 %\subtitle
\end{small}

\vskip 10mm
\begin{center}
\begin{large}Benedict Baur \end{large}   \\
\textit{\begin{small}baur@mathematik.uni-kl.de \\
\end{small} }
\begin{large} Florian Conrad \end{large}   \\
\textit{\begin{small}fconrad@mathematik.uni-kl.de \\
\end{small} }
\begin{large} Martin Grothaus \end{large} \\   
\textit{\begin{small}grothaus@mathematik.uni-kl.de \\
\end{small} }

\bigskip
\begin{small}
\textit{Functional Analysis and Stochastic Analysis Group, \\
Department of Mathematics, \\
University of Kaiserslautern, 67653 Kaiserslautern, Germany \\}
\end{small}

\end{center}

\textbf{Keywords :} Smooth contractive extension operator, elliptic differential operator, Feynman formula, Chernoff theorem.

\end{center}
\begin{abstract}
We prove two assumptions made in an article by Ya.A. Butko, M. Grothaus, O.G. Smolyanov concerning the existence of a strongly continuous operator semigroup solving a Cauchy-Dirichlet problem for an elliptic differential operator in a bounded domain and the existence of a smooth contractive embedding of a core of the generator of the semigroup into the space $C_c^{2,\alpha}(\R^n)$. Based on these assumptions a Feynman formula for the solution of the Cauchy-Dirichlet problem is constructed in the article mentioned above. In this article we show that the assumptions are fulfilled for domains with $C^{4,\alpha}$-smooth boundary and coefficients in $C^{2,\alpha}$.
\end{abstract}

\section{Introduction}
For a second order elliptic differential operator $L$ with Hölder continuous coefficients (see Definition \ref{DefDiffOp}) and a bounded domain $\Om \subset \R^n$, $n \in \N$, with certain assumptions on the boundary $\partial \Om$ we consider the Cauchy-Dirichlet problem:
For $u_0 \in C_0(\Omclo)$ sufficiently smooth find a function $u: [0,\infty) \to (C_0(\Omclo),\Vert \cdot \Vert_{\sup})$ differentiable in $t$ such that
\begin{align}
\frac{\partial u}{\partial t}(t,x) &= Lu(t,x), &t > 0, \, x \in \Om, \nonumber\\
u(0, x)&= u_0 (x), &x \in \overline{\Om}, \label{CDP} \\
u(t, x)&= 0, &t \ge 0, \, x \in \partial \Om \nonumber.
\end{align}
Let $(C_0(\Omclo),\Vert \cdot \Vert_{C_0(\Omclo)})$ be the Banach space of continuous functions vanishing at the boundary endowed with the norm of uniform convergence (also called supremum norm).
We define $(L,D(L))$ on $C_0(\Omclo)$ by:
% Definition of the differential operator
\begin{definition} \label{DefDiffOp} 
$$ Lu := \sum^n_{i,j=1} a_{ij} \pdi \pdj u + \sum^n_{i=1} b_i \pdi u + c u, \quad u \in D(L), $$
$$D(L) = \{ u \in C^{2,\alpha}(\Omclo) \, | \, u = Lu = 0 \ton \partial \Om \} \label{DefDL}.$$
\end{definition}
We assume the coefficients $a_{ij}$,$b_i$,$c$, $1 \le i,j \le n$ to be at least $C^{0,\alpha}(\Omclo)$-smooth and bounded by a constant $0<\coeffbound < \infty$. The matrix $A:=(a_{ij})_{i,j}$ is assumed to be symmetric and uniformly elliptic with ellipticity constant $\ellip>0$. Throughout this paper $\alpha$ denotes an arbitrary but fixed real number with $0 < \alpha < 1$.
Here as usual $C^{2,\alpha}(\Omclo)$ denotes the space of twice Hölder continuously differentiable functions such that the derivatives admit a (Hölder) continuous extension to the boundary. For a boundary point $x_0$, $Lu (x_0)$ is defined using the continuous extensions of the derivatives of $u$ and the coefficients of $L$ to $\partial \Om$. 

% The formula 
A so called \textit{Feynman formula} gives an approximation of the solution of \eqref{CDP} in terms of an iterated sequence of integrals over elementary functions only, see Definition \ref{DefF} below. In particular this formula gives a finite-dimensional approximation to the well-known Feynman-Kac formula, see \cite{ZhJ}:
\begin{equation}\label{FKFPara}
u(t,x)=\mathbb{E}_x \bigg[\exp\bigg(\int_0^t
c(\xi_\tau)d\tau\bigg)u_0(\xi_t )\,\bigg| \,\,  t<\tau_\Om \bigg],
\quad x \in \Om, \, t > 0.
\end{equation}
Here $\mathbb E_x$ denotes the expectation w.r.t the law of the diffusion process with diffusion matrix $a_{ij}$ and drift coefficient $b_i$ starting in $x$, according to \cite[Theo. 3.1]{ZhJ}.

% Hier besser ersetzen durch klassische Feyn-Kac formula for parabolic equations
We now recall the approximation formula from \cite{BGS} and some main steps in its proof to motivate the assumptions which we will prove in this paper.
Define for $u \in D(L)$ 
\begin{multline}
F_t u (x) := \\
\frac{\psi_{s(t)}(x) \, \exp(t c(x))}{\sqrt{a(x)(4 \pi t)^n}} \int_{\R^n} \exp \left( - \frac{\langle A^{-1}(x)(x-y + t b(x)),x-y + t b(x)\rangle }{4t} \right) \hspace{-4pt} Eu(y) dy \\ 
= \hspace{-1pt} \frac{\psi_{s(t)}(x) \, \exp(t v(x))}{\sqrt{a(x)(4 \pi t)^n}} \hspace{-2pt} \int_{\R^n} \hspace{-2pt} \exp \left( - \frac{\langle A^{-1}(x)(x-y),x-y \rangle }{4t} + \frac{1}{2} \langle A^{-1}b(x),x-y \rangle \right) \hspace{-4pt}  Eu(y) dy. \label{DefF}
\end{multline}
Here $\langle \cdot, \cdot \rangle$ denotes the Euclidean scalar product on $\R^n$, $a(x)$ denotes the determinant of $A(x)$ and $v(x):=c(x) - \frac{1}{4} \langle A^{-1}b(x),b(x) \rangle $. 
$\psi_{s(t)}$ is a family of cutoff functions with compact supports in $\Om$, defined in \cite{BGS} before Lemma 4.3.
% \begin{definition} \label{DefPsi}
% Let $s:[0,\infty) \to [0,\infty)$ be a $C^{\infty}$-smooth function which monotonically decreases to $0$ as $t \searrow 0$ such that $s(t) = o(t)$. Define $\Om_{s(t)} := \{ x \in \Om \ | \ \dist(x,\partial \Om) > s(t) \}$. Let \textit{$\psi_{s(t)}$} be a cutoff function for $\Om_{s(t)}$ in $\Om$.
% \end{definition}
Moreover $E$ is a suitable extension operator embedding $D(L)$ into $C^{2,\alpha}_c(\R^n)$, see Assumption \ref{assEmb} below.\\
Note that we have defined $L$ without the factor $\frac{1}{2}$ in front of the second order terms, which leads to a slightly different form of $F_t$ than in \cite{BGS}.\\
% The integral kernel is the Gausskernel with covariant matrix $t A(x)$ translated by the drift vector $B$. This corresponds to a particle with fixed diffusion matrix $A$ and drift vector $B$. The kernel is convoluted with the smooth extension of $u$ on $\R^n$. During the construction of the extension operator it turned out that this operator corresponds to the absorbing of the particle at the boundary. Finally the integral is multiplied with $exp(tC)$ approximating the weighting in the Feynman-Kac formula. Furthermore a suitable cutoff is applied to ensure that the function $F_t u$ has zero boundary values. 
The subindex $c$ denotes that the functions in  $C_c^{2,\alpha}(\Omclo)$ have compact support in $\Omclo$, analogously the subindex $0$ denotes that the functions vanish at the boundary. The analogous notation is used for the spaces $C^{k,\alpha}(\Omclo)$. Using a Taylor expansion of $Eu$ it can be shown that for $u_0 \in C^{2,\alpha}_0(\Omclo)$ with $Lu=0$ on $\partial \Om$: 
\begin{align} F_t u_0 = u_0 + t L u_0 + o(t), \label{Consistency} \end{align} 
with $o(t)$ independent of $x$, see \cite[Lemma 4.1, 4.2 and 4.3]{BGS}. So $F_t$ approximates the solution to the Cauchy-Dirichlet problem $\eqref{CDP}$ for small $t$ and one might ask whether the solution for $t>0$ can be obtained by splitting $[0,t]$ in small time intervals and applying $F_{t/n}$ in each interval, i.e 
$$ u(t) = \limn (F(t/n))^n u_0.$$

A well-known tool to prove convergence is the Chernoff theorem for strongly continuous operator
semigroups, see \cite[Theo. 2.2]{BGS} or \cite[Theo. 5.2]{EN}.
\begin{theorem}[Chernoff theorem]
Let $X$ be a Banach space, $F:[0,\infty)\to{L}(X)$ a continuous
mapping such that $F(0) =Id$ and $\|F(t)\|\le \exp(at)$ for
some $ a\in [0, \infty)$ and all $t \ge 0$. Let $D$ be a linear
subspace of $D(F'(0))$ such that the restriction of the operator
$F'(0)$ to this subspace is closable. Denote by $(\overline{L}, D(\overline{L}))$ the
closure. If $(\overline{L}, D(\overline{L}))$ is the generator of a strongly continuous
semigroup $(T_t)_{t \ge 0}$, then for any $0 \le t_0 < \infty$ the
sequence $\big(\big(F(t/n)\big)^n\big)_{n \in {\mathbb N}}$
converges to $T_t$ as $n\to\infty$ in operator norm, uniformly with
respect to $t\in[0,t_0]$, i.e., $T_t =
\lim_{n\to\infty}\big(F(t/n)\big)^n$
locally uniformly in $[0, \infty)$.
\end{theorem}

% Chernoff rein oder nicht rein ?
To ensure that $(F_t)_{t \ge 0}$ defined above is uniformly exponentially bounded the extension operator $E$ is assumed to be contractive w.r.t the sup norm.
Note that an estimate of the form $\Vert F_t \Vert \le M \exp(a t)$ for $M>1$ is not sufficient to apply the Chernoff theorem. So boundedness of $E$ would not be sufficient.
This leads to the following assumption:
\begin{assumption} \label{assEmb}
There exists a linear embedding $E: D(L) \to C^{2,\alpha}_c(\R^n)$ with the properties:
\begin{enumerate}
\item $Eu |_{\Omclo} = u$
\item $\underset{x \in \R^n}{\sup} |Eu(x)|  = \underset{x \in \Om}{\sup} |u(x)|$.
\end{enumerate}
\end{assumption}
Here $C_c^{2,\alpha}(\R^n)$ denotes the space of twice hölder continuously differentiable functions with compact support in $\R^n$.

Moreover to apply the Chernoff theorem the solution to $\eqref{CDP}$ must be represented by a strongly continuous operator semigroup and \eqref{Consistency} must hold on a core of the generator of the semigroup. This leads to the following assumption:
\begin{assumption} \label{assCore}
Let $(L,D(L))$ be as in Definition \ref{DefDiffOp}. Assume that $(L,D(L))$ is closable in $(C_0(\Omclo),\Vert \cdot \Vert_{C_0(\Omclo)})$ and the closure $(\overline{L},D(\overline{L}))$ generates a strongly continuous operator semigroup. 
\end{assumption}
Note that then $D(L)$ is a core of the generator of the operator semigroup for the Cauchy-Dirichlet problem \eqref{CDP}. This assumption corresponds to \cite[Ass.~3.2]{BGS}.\\
% Assumption 2 Core

% Assumption 1 Continuation
% Property 1 means that the functions are extended beyond the boundary in a smooth way, since $E$ is required to map into $C^{2,\alpha}_c(\R^n)$. Property 2 means that the embedding is contractive (w.r.t the sup Norm).
% \begin{remark}
% This assumption is slightly weaker than the assumption [\cite{BGS},Ass 3.1], but for the proof of the convergence of the formula it is sufficient to construct the embedding operator just on $D(L)$ and not on $C_0^{2,\alpha}$, see \ref{RemEmb}.
% \end{remark}

Assuming \ref{assCore} and \ref{assEmb}, the following theorem is proved in \cite[Theo.~4.5]{BGS}:
\begin{theorem} \label{ConvgFormula}
Let $F(t)$ be as in Definition \ref{DefF} and $(T_t)_{t \ge 0}$ the semigroup generated by $(\overline{L},D(\overline{L}))$ (due to \ref{assCore}).
Then for all $t \ge 0$ it holds
$$ T_t = \limn \big(F(t/n)\big)^n $$
w.r.t the operator norm.
\end{theorem}
Note that the proof of \cite[Theo. 4.5]{BGS} is based on \cite[Lemma 4.2]{BGS}, where the existence of a smooth contractive embedding operator for functions in $C_0^{2,\alpha}(\Omclo)$ is assumed. However this embedding operator is applied only to functions in $D(L)$, so the weaker Assumption \ref{assEmb} is also sufficient.

The aim of the present paper is to prove Assumption \ref{assCore} and \ref{assEmb} under conditions on the smoothness of the coefficients of $L$ and the boundary of $\Om$.

In section 2 we prove Assumption \ref{assCore} for the case of $C^{0,\alpha}(\Omclo)$-smooth coefficients and domains $\Om$ which are $C^{2,\alpha}$-smooth and bounded, see Theorem \ref{SemigroupCoreBounded}.

In section 3 we prove Assumption \ref{assEmb} for coefficients in $C^{2,\alpha}(\Omclo)$ and domains $\Om$ which are $C^{4,\alpha}$-smooth and bounded, see Theorem \ref{EmbeddingWholeSpace}.

\begin{remark} For $C^{2,\alpha}$-smooth $\Om$ and $C^{0,\alpha}$-smooth coefficients the semigroup $(T_t)_{t \ge 0}$ is even analytic and the solution $u(t) = T_t u_0$ is in $C_0^{2,\alpha}(\Omclo)$, see Theorem \ref{DomClosureAnalytic}.
Moreover from \cite[Theo.~3.2]{ZhJ} it follows that the semigroup generated by $(\overline{L},D(\overline{L}))$ is represented by the Feynman-Kac formula \eqref{FKFPara}.
The Feynman-Kac formula holds also under weaker conditions on the boundary, but then $C_0(\Omclo)$ has to be replaced by a larger space, see \cite[Theo.~3.3]{ZhJ}.
\end{remark}

\section{Existence and Regularity}
First we state two well-known theorems concerning elliptic differential operators of second order.

\begin{lemma} \label{MaximumPrinciple}
Let $u \in C^2_0(\Omclo)$, $L$ as in Definition $\ref{DefDiffOp}$. If $u$ attains its maximum (minimum) at an interior point $x_0$ of $\Om$ then for $\lambda_0 := \sup_{x \in \Om} c(x)$ it holds:
\begin{align} (Lu -\lambda_0 u)(x_0) \le 0 \, (\ge 0). \label{MaxPrincip} \end{align}
In particular, the operator $L - \lambda_0$ is dissipative on $C_0(\Omclo)$.
\end{lemma}
\begin{proof}
The proof of $\eqref{MaxPrincip}$ can be found in the proof of \cite[Theo. 3.1]{GilTru}. Let $x_0 \in \Om$ be a point where the supremum of $|u|$ is attained, then for the bounded linear functional $F:C_0(\Omclo) \to \R$, $v \mapsto \textnormal{sgn}(u(x_0)) v(x_0)$ it holds: $F(u) = \Vert u \Vert_{C_0(\Omclo)}$ and by the statement above: $ F((L-\lambda_0)u) = (Lu-\lambda_0 u) (x_0) \le 0$. So $L-\lambda_0$ is dissipative.
\end{proof}

\begin{theorem} \label{EllipSol}
Let $L$ be an elliptic differential operator with coefficients $a_{ij}$,$b_i$,$c$ in $C^{0,\alpha}(\Omclo)$ and $c \le 0$. Let further $\Om$ be a bounded $C^{2,\alpha}$-smooth domain. Then for $f \in C^{0,\alpha}(\Omclo)$ there exists a unique solution $u \in C_0^{2,\alpha}(\Omclo)$ such that:
$$ Lu = f.$$ 
\end{theorem}
\begin{proof} See \cite[Theo. 6.14]{GilTru}.
\end{proof}

\begin{theorem} \label{SemigroupCoreBounded}
Let $(L,D(L))$ as in Definition \ref{DefDiffOp} with $C^{0,\alpha}(\Omclo)$-smooth coefficients, $\Om$ be a bounded $C^{2,\alpha}$-smooth domain. Then the closure of $(L,D(L))$ in $(C_0(\Omclo),\Vert \cdot \Vert_{C_0(\Omclo)})$ generates a strongly continuous operator semigroup.
\end{theorem}
\begin{proof}
Set $\lambda_0 := \sup_{x \in \Omclo} \, c(x)$.
Then the operator $L - \lambda_0$ is dissipative by Lemma \ref{MaximumPrinciple} and densely defined. Thus $(L,D(L))$ is closable. Moreover since for $\lambda > \lambda_0$, it holds $\tilde{c} = c - \lambda < 0$,  Theorem \ref{EllipSol} applies. So the operator $L-\lambda$ has dense range for all $\lambda > \lambda_0$. Thus the closure $(\overline{L},D(\overline{L}))$ of $(L,D(L))$ generates a strongly continuous semigroup.
\end{proof}
\begin{remark} If $c \le 0$ the operator semigroup is contractive, otherwise the growth bound is given by $\exp(\lambda_0 t)$, where $\lambda_0 =\sup_{x \in \Omclo} \, c(x)$ is as in the proof of Theorem \ref{SemigroupCoreBounded}.
\end{remark}

\begin{remark}
By the previous theorem we get that $D(L)$ is a core for the generator of the semigroup corresponding to the Cauchy-Dirichlet problem $\eqref{CDP}$. The elements in the domain of $D(\overline{L})$ need not to be twice continuously differentiable and so solutions obtained by the operator semigroup at first sight need not to be  classical solutions. However the following theorem shows that functions in $D(\overline{L})$ are twice weakly differentiable and $T_t u$ is even in $C^{2,\alpha}(\Omclo)$ for $t > 0$. 
\end{remark}

Using the results of \cite{Lunardi}, we get:
\begin{theorem} \label{DomClosureAnalytic}
In the situation as in Theorem \ref{SemigroupCoreBounded} we have:
\begin{enumerate}
\item For the domain $D(\overline{L})$ it holds: \\ \begin{align}D(\overline{L}) = \big \{ u \in \bigcap_{p\ge1} W^{2,p}_{loc}(\Om) \, | \, Lu \in C(\Omclo), u \in C_0(\Omclo) \big \}. \label{Dpr} \end{align} 
\item The corresponding semigroup $(T_t)_{t \ge 0}$ is the restriction of an analytic semigroup.
\item $T_t u_0 \in C_0^{2,\alpha}(\Omclo)$ for $t>0$ and $u_0 \in C_0(\Omclo)$.
\end{enumerate}
\end{theorem}
\begin{proof}
Let $D'$ be the RHS of $\eqref{Dpr}$.
By \cite[Cor. 3.1.21(ii)]{Lunardi} the differential operator $L$ defined on $D'$ generates an analytic semigroup on $C_0(\Omclo)$. Since $D(L) \subset D'$ and $(L,D')$ is closed it follows $D(\overline{L}) \subset D'$. But $(\overline{L},D(\overline{L}))$ generates a semigroup, thus  $D(\overline{L})$ cannot be a proper subset of $D'$ by the Hille--Yosida theorem.
In the notation of \cite[Theo. 5.1.11]{Lunardi} we have $T_t = \exp(t L)$, $u:= \exp(t L)u_0 = T_t u_0$ and $f=0$. Then the last statement follows from \cite[Theo. 5.1.13(iv)]{Lunardi}. %See also \cite[Theo. 3.1.35]{Lunardi} for the characterization of $D(L)$ as an interpolation space.
\end{proof}

\section{Embedding operator}
In this section we construct the contractive smooth embedding of $D(L) \subset C_0^{2,\alpha}(\Omclo)$ into the space $C^{2,\alpha}_c(\R^n)$. We emphasize the requirement, that the supremum norm of the continued function is not increased. Due to this requirement usual extension operators, like in \cite[Sec. 6.9]{GilTru}, are not suitable, since they increase in general the supremum norm. One possibility to continue a function is to do a reflection at the boundary. That is each point outside corresponds to a point inside $\Om$ and the function at the point outside is defined to be the value at the corresponding point inside multiplied by $1$ or $-1$. Such an extension is clearly contractive, however it is not smooth enough in general. For example take $\Om=\R^+$, $u \in C^{2,\alpha}_0(\R^+_0)$. Define $\tilde{u}(x) = -u(-x)$ for $x<0$. Then since $u(0)=0$ and $\pdx \tilde{u}(x) = \pdx u(-x)$, $\tilde{u}$ is a continuously differentiable continuation of $\tilde{u}$. However $\pdx \pdx \tilde{u}(x) = - \pdx \pdx u(-x)$. So $\tilde{u}$ is in general not $C^{2}(\R)$ smooth. On the other hand, if additionally $\pdx \pdx u(0) = 0$, then $\tilde{u}$ is $C^2(\R)$, and for $u \in C^{2,\alpha}(\R^+_0)$ it follows $\tilde{u} \in C^{2,\alpha}(\R)$. Here $\pdx \pdx u(0)$ means the continuous extension of $\pdx \pdx u$ to the boundary. So we get a smooth and contractive continuation if we restrict ourselves to the subspace of $C^{2,\alpha}_0(\R^+_0)$ with the additional condition $\pdx \pdx u(0) = 0$.
As a motivating example we generalize this this construction to functions with boundary condition $a \pdx \pdx u(0) + b \pdx u(0)= u(0)=0$. In this case the reflection has to be replaced by a squeezed reflection, see Theorem \ref{EmbLaplaceHalfspaceDrift}.
Then we give the construction of the embedding operator for $C^{4,\alpha}$-smooth domains and elliptic differential operators with $C^{2,\alpha}(\Omclo)$-smooth coefficients. In this case the reflection has to be done along a certain direction, see Lemma \ref{ReflDirection}. To ensure the $C^{2,\alpha}$-smoothness of the continued function, we need that the direction of reflection depends $C^{2,\alpha}$-smooth on the boundary point. We first construct a local extension in Theorem \ref{LocalEmbedding} and then the global one in Theorem \ref{EmbeddingWholeSpace}.
% The construction is a modification of what the author presented in \cite{Baur}.
\subsection{Half-line}
% \begin{theorem} \label{EmbLaplaceHalfspace}
% Let $L = \pdx \pdx$, $D(L) = \{ u \in C^{2,\alpha}(\Omclo) \, | \, u(0) = Lu(0) = 0 \}$. Then $D(L)$ can be embedded into $C^{2,\alpha}_c(\R^n)$ such that $\Vert E u \Vert_{C_0(\R^n)} = \Vert u \Vert_{C_0(\Omclo)}$.
% \end{theorem}
% 
% \begin{proof}
% Define the extension $Eu$ of $u$ by:
% $$
% Eu (x) =  
% \begin{cases}
%   u(x) \, & \, x \ge 0 \\
%   -u(-x) \, & \, x < 0 \\ 
% \end{cases}.
% $$
% 
% Then: ...
% \end{proof}

\begin{theorem} \label{EmbLaplaceHalfspaceDrift}
Let $L = a \pdx \pdx + b \pdx$ with $a > 0$, $b \in \R$, $D(L) = \{ u \in C^{2,\alpha}(\R^+_0) \, | \, u(0) = Lu(0) = 0, \, u \, \text{is bounded.} \, \}$. Then there exists an $\varepsilon>0$ and an embedding $E: D(L) \to C^{2,\alpha}([-\varepsilon,\infty))$ with $Eu|_{\R^+}=u$ and
\begin{align}\underset{x \in [-\varepsilon,\infty)}{\sup} |Eu(x)| = \underset{x \in [0,\infty)}{\sup} |u(x)|. \label{ContrHalf} \end{align}
\end{theorem}

\begin{proof}
Define for $x<0$ $F(x) = -x + \frac{b}{a} x^2$. Then there exists an $\varepsilon > 0$ such that $F(x) > 0$ for $-\varepsilon < x < 0$.
Define the extension $E_0 u$ of $u$ by:
$$
E_0 u (x) =  
\begin{cases}
  u(x) \, & \, x \ge 0 \\
  -u(F(x)) \, & \, - \varepsilon < x < 0 \\ 
\end{cases}.
$$
By construction condition $\eqref{ContrHalf}$ is fulfilled. 
Moreover for $y <0$ we have:
\begin{align*}
\pdx E_0 u (y) = - \pdx u(F(y))(\pdx F(y)) = - \pdx u(F(y)) \left(-1 + 2\frac{b}{a}y \right),
\end{align*}
\begin{multline*}
\pdx^2 E_0 u (y) = - \pdx^2 u(F(y)) (\pdx F(y))^2 - \pdx u(F(y))(\pdx^2 F(y)) \\ 
= - \pdx^2 u(F(y)) \left(-1 +2 \frac{b}{a} y\right)^2 - \pdx u(F(y))\left(2 \frac{b}{a}\right).
\end{multline*}

Using the continuity of $F$ we get that for $y \nearrow 0$:
$$ \pdx E_0 u(y) \to \pdx u(0). $$
Here $\pdx u(0)$ denotes the continuous extension of $\pdx u$ from $\R^+$ to $0$. The same notation is used for $\pdx \pdx u(0)$. 

For the second derivative we have for $y \nearrow 0$
$$ \pdx^2 E_0 u(y) \to - \pdx^2 u(0) - \pdx u(0) \left(2 \frac{b}{a}\right). $$

Since $u \in D(L)$, we have $Lu(0) = a \pdx^2 u(0) + b \pdx u(0) = 0$. Thus $ -\frac{b}{a} \pdx u(0) = \pdx^2 u(0) $
so
$$ - \pdx^2 u(0) - \pdx u(0) \left(2 \frac{b}{a}\right) = \pdx^2 u(0). $$
Thus also
$$ \pdx^2 E_0 u(y) \to \pdx^2 u(0) \, \text{as} \, y \nearrow 0.$$

By construction $E_0 u$ is twice Hölder continuously differentiable in $\R^+$ and $(-\varepsilon,0)$, moreover by the calculations above the extensions of the interior derivatives in $\R^+$ and $(-\varepsilon,0)$ to $0$ coincide. Thus $E_0 u \in C^{2,\alpha}((-\varepsilon,\infty))$.
Choose now a cutoff $\eta$ for $\R^+_0$ in $(-\varepsilon,\infty)$. Define $Eu(x) := \cutoff(x) Eu$. Then $Eu \in C^{2,\alpha}([-\varepsilon,\infty))$.
\end{proof}

\subsection{General smooth domain}

\begin{definition} \label{DefDiffeomorphism}
A domain $\Om \subset \R^n$ is called \textit{$C^{k,\alpha}$-smooth} ($k \in \N$, $\, 0<\alpha<1$), if there exists for each point $x_0 \in \partial \Om$ a neighborhood $V$ of $x_0$, a neighborhood $U$ of $0$ and a $C^{k,\alpha}$-smooth  diffeomorphism $\psi: U \to V$ such that:
\begin{enumerate}
\item $\psi(U \cap \R^n_0) = (\partial \Om \cap V)$
 \item $\psi(U \cap \R^n_+) = (\Om \cap V).$ 
\end{enumerate}

Here $\R^n_0$ denotes the $n-1$-dimensional hyperplane $\{ x \in \R^n \, | \, x_n=0\}$ and $\R^n_+$ the halfspace $\{ x \in \R^n \, | \, x_n > 0 \}$.
\end{definition}

% For the ease of notation we drop the factor $\frac{1}{2}$ in the differential operator now:
% 
% \begin{definition} \label{DefDiffOp2} 
% $$ L : D(L) \to C_0(\Omclo), $$
% $$ u \mapsto \sum^n_{i,j=1} a_{ij} \pdi \pdj u + \sum^n_{i=1} b_i \pdi u + c u. $$
% $$D(L) = \{ u \in C^{2,\alpha}(\Omclo) \, | \, u = Lu = 0 \ton \partial \Om \}, \label{DefDL}$$.
% \end{definition}

\begin{lemma} \label{TrafoDiffOp}
Let $L$ be a differential operator as in Definition \ref{DefDiffOp} with $C^{2,\alpha}$-smooth coefficients and $\Om$ a $C^{4,\alpha}$-smooth domain. Let $x \in \partial \Om$, $V$ be a neighborhood of $x$ and $v_1,...,v_n :V \to \R^n$ a family of $C^{3,\alpha}$-smooth normalized vectorfields, which are pairwise orthogonal to each other. Then for all points in $V \cap \Omclo$ $L$ can be written in partial derivatives in directions along the vector fields $v_1,...,v_n$ with $C^{2,\alpha}$-smooth coefficients. In particular the first order coefficient of $\partial_{v_n}$ has the form
\begin{align} \tilde{b} := \langle b, v_n \rangle + \sum^n_{l=1}  \langle v_l, A \partial_{v_l} v_n \rangle . \label{EqBtilde} \end{align}
\end{lemma} 
\begin{proof}
Denote by $e_i$ the i-th unit vector.
Since $e_i = \sum_{j=1}^n \langle e_i, v_j \rangle v_j$ we have
\begin{align}
\pdi = \sum_{j=1}^n \langle e_i, v_j \rangle \partial_{v_j} \label{PdiLocal} 
\end{align}
and 
\begin{multline}
\pdi \pdj = \left(\sum_{k=1}^n \langle e_i, v_k \rangle \partial_{v_k}\right) \left(\sum_{l=1}^n \langle e_j, v_l \rangle \partial_{v_l} \right) \\
= \sum_{k=1}^n \sum_{l=1}^n \langle e_i, v_k \rangle \langle e_j, v_l \rangle \partial_{v_k} \partial_{v_l} + \langle e_i, v_k \rangle \langle e_j, \partial_{v_k} v_l \rangle \partial_{v_l}. \label{PdijLocal}
\end{multline}

Plugging \eqref{PdiLocal}, \eqref{PdijLocal} into $L$ yields:

\begin{multline}
L = \sum^n_{i=1} \sum^n_{j=1} a_{ij} \pdi \pdj + \sum^n_{i=1} b_i \pdi + c = 
\\
\sum^n_{i,j,k,l=1} a_{ij} \big(\langle e_i, v_k \rangle \langle e_j, v_l \rangle \partial_{v_k} \partial_{v_l} + \langle e_i, v_k \rangle \langle e_j, \partial_{v_k} v_l \rangle \partial_{v_l} \big) + \sum^n_{i=1} b_i \sum_{k=1}^n \langle e_i, v_k \rangle \partial_{v_k} + c \\
= \sum^n_{k,l} \langle v_k,A v_l \rangle \partial_{v_k} \partial_{v_l} + \sum^n_{k=1} \left (\langle b, v_k \rangle + \sum^n_{l=1} \langle v_l, A \partial_{v_l} v_k \rangle \right) \partial_{v_k} + c 
\end{multline}

Since $v_k \in C^{3,\alpha}(V)$, the second order coefficients $\langle v_k, A v_l \rangle$ and the first order coefficients $\langle b, v_k \rangle + \sum_{l=1}^n \langle v_l, A \partial_{v_l} v_k \rangle$ are $C^{2,\alpha}$-smooth.
\end{proof}
% Observe that the drift coefficient in the local coordinates consists of derivatives of the $C^{3,\alpha}$ vector fields. Since the construction of the embedding operator requires $C^{2,\alpha}$ smooth drift coefficients the domain must be assumed to be $C^{4,\alpha}$. 

Now we construct the direction of reflection.
\begin{lemma} \label{ReflDirection}
Let $L$ be a differential operator as in Definition \ref{DefDiffOp} with $C^{2,\alpha}$-smooth coefficients, $\Om$ a $C^{4,\alpha}$-smooth domain. Then for every point $x_0 \in \partial \Om$ and $V$ the neighborhood of Definition \ref{DefDiffeomorphism} there exist a $C^{2,\alpha}$-smooth vector field $\tilde{v}_n : V \cap \partial \Om \to \R^n$ pointing into $\Om$ and $C^{2,\alpha}$-smooth coefficients $\tilde{a}_{nn}$, $\tilde{b} : V \cap \partial \Om \to \R$ such that for $u \in D(L)$ as in Definition \ref{DefDL} and $x \in V \cap \partial \Om$:
\begin{eqnarray} \tilde{a}_{nn} \pdnvtilde \pdnvtilde u (x) + \tilde{b} \pdnvtilde u (x)= 0. \label{ReflDirectionEq} 
\end{eqnarray}
In particular, it holds $\tilde{a}_{nn} = \langle v_n, A v_n \rangle$ and $\tilde{b}$ is as in $\eqref{EqBtilde}$.
Here $\pdnvtilde u(x)$ and $\pdnvtilde \pdnvtilde u(x)$ are the continuous extension of the inner derivatives to the boundary.
\end{lemma}
\begin{proof}
Let $x_0 \in \partial \Om$, $V$ be the neighborhood of $x_0$, $U$ be the neighborhood of $0$ and $\psi$ be the  $C^{4,\alpha}$-smooth diffeormophism of Definition \ref{DefDiffeomorphism}. By orthonormalising the column vectors of the Jacobi matrix $D \psi$ we obtain a family of vector fields $v_1,...,v_n : V \to \R^n$. Since $\psi$ maps points $z \in U$ with $z_n > 0$ into $\Om$ the last column vector $(D \psi)_n$ on $\partial \Om$ points into $\Om$ and thus also $v_n$ restricted to $\partial \Om$ points into $\Om$.
Note that $v_1,...,v_n$ fulfill the assumption of $\ref{TrafoDiffOp}$.
 
Define
$$ \tilde{v}_n := v_n + \sum^{n-1}_{i=1} \frac{ \langle v_i, A v_n \rangle}{\langle v_n, A v_n \rangle} v_i. $$
Then $\tilde{v}_n \in C^{3,\alpha}(\partial \Om)$.
Since $u=0$ on the boundary, we have
\begin{align} \partial_{v_i} u & = 0, \label{BondPdi} \\
\partial_{v_i} \partial_{v_j} u & = 0 \, \quad \text{on} \, V \cap \partial \Om \nonumber
\end{align}
for $i,j \ne n$.

Therefore we have
\begin{align}
\partial_{v_n} u & = \pdnvtilde u, \label{FirstPdn} \\
\partial_{v_i} \partial_{v_n} u & = \partial_{v_i} \pdnvtilde u  \, \quad \text{on} \, V \cap \partial \Om \, \nonumber 
\end{align}
for $i \ne n$.
Furthermore:
\begin{multline}
\partial_{v_n} \partial_{v_n} u = \left(\pdnvtilde - \sum^{n-1}_{i=1} \frac{ \langle v_i, A v_n \rangle}{\langle v_n, A v_n \rangle} \partial_{v_i}  \right) \left(\pdnvtilde - \sum^{n-1}_{i=1} \frac{ \langle v_i, A v_n \rangle}{\langle v_n, A v_n \rangle} \partial_{v_i}  \right) u \\
= \pdnvtilde \pdnvtilde u - 2 \sum^{n-1}_{i=1} \frac{ \langle v_i, A v_n \rangle}{\langle v_n, A v_n \rangle} \partial_{v_i} \pdnvtilde u \\
= \pdnvtilde \pdnvtilde u - 2 \sum^{n-1}_{i=1} \frac{ \langle v_i, A v_n \rangle}{\langle v_n, A v_n \rangle} \partial_{v_i} \partial_{v_n} u \quad \text{on} \, V \cap \partial \Om. \label{SecPdn}
\end{multline}

Let $\tilde{b}$ be as in $\eqref{EqBtilde}$.
The boundary condition $Lu = 0$ implies using \eqref{BondPdi}, \eqref{FirstPdn}, \eqref{SecPdn} and Lemma \ref{TrafoDiffOp}:
\begin{multline*}
0 = Lu = \langle v_n, A v_n \rangle \partial_{v_n} \partial_{v_n} u + 2 \sum^{n-1}_{i=1} \langle v_i, A v_n \rangle \partial_{v_i} \partial_{v_n} u + \tilde{b} \partial_{v_n} u \\
= \langle v_n, A v_n \rangle \left(\pdnvtilde \pdnvtilde u - 2 \sum^{n-1}_{i=1} \frac{ \langle v_i, A v_n \rangle}{\langle v_n, A v_n \rangle} \partial_{v_i} \partial_{v_n} u \right) + 2 \sum^{n-1}_{i=1} \langle v_i, A v_n \rangle \partial_{v_i} \partial_{v_n} u + \tilde{b} \pdnvtilde u \\
= \langle v_n,A v_n \rangle \pdnvtilde \pdnvtilde u + \tilde{b} \pdnvtilde u \quad \text{on} \, V \cap \partial \Om.
\end{multline*}

\end{proof}
We construct now a local embedding operator.
For a point $x \in \R^n$ we denote by $x'$ the vector of the first $n-1$ coordinates. 
\begin{theorem} \label{LocDiffeo}
Let $L$ be a differential operator as in Definition \ref{DefDiffOp} with $C^{2,\alpha}$-smooth coefficients, $\Om$ a $C^{4,\alpha}$-smooth domain. Then for each $x_0 \in \partial \Om$, there exists a neighborhood $\tilde{V}$ of $x_0$, a neighborhood $\tilde{U}$ of  $0$ and a $C^{2,\alpha}$-smooth diffeomorphism $\tilde{\psi}$ as in Definition \ref{DefDiffeomorphism} with the additional property:
\begin{align}\pdn (u \circ \tilde{\psi}) (z) = \pdnvtilde u (\tilde{\psi}( z)), \quad z \in U \cap \R^n_0, \ \end{align}
where $\tilde{v}_n : V \to \R^n$ is the vector field provided by Lemma \ref{ReflDirection}.
\end{theorem}
\begin{proof}
Let $x_0 \in \partial \Om$, $\psi$, $U$, $V$ the $C^{4,\alpha}$-diffeomorphism and the neighborhoods as in Definition \ref{DefDiffeomorphism}. Furthermore let $\tilde v_n : V \to \R^n$ the $C^{2,\alpha}$-smooth vector field provided by Lemma $\ref{ReflDirection}$.
For $x \in V \cap \partial \Om$, $z:= \psi^{-1}(x)$ define
$$ y^{(n)}(z) = (D \psi(z))^{-1} \tilde{v}_n(x). $$
Since $\tilde{v}_n(x)$ points into $\Om$ and is not an element of the tangential space which is spanned by the first $(n-1)$ column vectors of $D \psi(z)$, it holds: $y_n^{(n)}(z) > 0$. Choose now $\varepsilon>0$ sufficiently small such that $U' := [-\varepsilon,\varepsilon]^n \subset U$. 
Define $\psi_0 : U' \to \R^n$ by
$$ (w',w_n) \mapsto w' + w_n y^{(n)}(w',0), $$
then:
$$ D \psi_0 (w',0) = 
\left(\begin{array}{ll}
\begin{array}{l}
 \Id_{(n-1) , (n-1)} \\
 0
\end{array}
& +y^{(n)}(w',0)
\end{array}
\right)
.$$

Since $y_n^{(n)}(z)>0$, $\text{Det} \, D \psi_0(w',0) \ne 0$ and $\psi_0(0)=0$. So there exists a neighborhood of zero $\tilde{U} \subset U'$ such that $\psi_0: \tilde{U} \to \psi_0(\tilde{U})$ is also a $C^{2,\alpha}$-smooth diffeomorphism. Choosing $\tilde{U}$ small enough we get $\psi_0 (\tilde{U}) \subset U$ and $\tilde{V}:= \psi \circ \psi_0 \, (\tilde{U})$ is a neighborhood of $x_0$. Define now $\tilde{\psi}:= \psi \circ \psi_0$, then
$\psi_0 (\tilde{U} \cap \R^n_0) \subset \R^n_0$ and $\psi_0 (\tilde{U} \cap \R^n_+) \subset \R^n_+$, $\psi_0 (\tilde{U} \cap \R^n_-) \subset \R^n_-$ together with the corresponding properties of $\psi$ imply:
$$ \tilde{\psi} (\tilde{U} \cap \R^n_+) = \Om \cap \tilde{V},$$
$$ \tilde{\psi} (\tilde{U} \cap \R^n_0) = \partial \Om \cap \tilde{V}.$$

Now let $z \in \tilde{U} \cap \R^n_0$, then $z = \psi_0(z)$, $x := \psi(z) \in \partial \Om$ and
\begin{align}
\pdn (u \circ \tilde \psi) (z) = \pdn (u \circ \psi \circ \psi_0) (z) = \nabla u(x) (D \psi(z)) y^{(n)}(z) = \nabla u(x) \tilde v_n(x) = \pdnvtilde u(x).
\end{align}

\end{proof}

\begin{theorem} \label{LocalEmbedding}
Let $(L,D(L))$ be a differential operator as in Definition \ref{DefDiffOp} with $C^{2,\alpha}$-smooth  coefficients, $\Om$ a $C^{4,\alpha}$-smooth domain. Then for each $x_0 \in \partial \Om$ there exists a neighborhood $\head{V}$ of $x_0$ and a linear bounded operator $E: D(L) \to C^{2,\alpha}(\head{V})$ such that for $u \in D(L)$ it holds $Eu |_{\head{V} \cap \Om} = u |_{\head{V} \cap \Om}$ and
\begin{align}\sup_{y \in \head{V}} | Eu(y) | = \sup_{y \in \head{V} \cap \Om} \, |u|. \label{ContrLoc} \end{align}

\end{theorem}

\begin{proof}
Let $x_0 \in \partial \Om$, $\tilde{V}$, $\tilde{U}$ be the neighborhoods and $\tilde{\psi}$ the diffeomorphism provided by Theorem \ref{LocDiffeo}. 
Define $a' := \tilde a_{nn} \circ \tilde \psi$, $b' := \tilde{b} \circ \tilde \psi$ with $\tilde{a}_{nn}$, $\tilde{b}$ are as in \eqref{ReflDirectionEq}.
Choose $\varepsilon>0$ such that $[-\varepsilon,\varepsilon]^n \subset \tilde{U}$.
Define: $$F : [-\varepsilon,\varepsilon]^n  \cap \{ z_n < 0 \} \to \R^n \, ,$$ 
$$ z \mapsto \left(z', - z_n + \frac{b'(z',0)}{a'(z',0)} z^2_n \right). $$

For the derivatives of $F$ we have:
\begin{eqnarray} \pdn F_n(z) = \left(-1 + 2 z_n \frac{b'(z',0)}{a'(z',0)}\right), \nonumber \\
\pdn \pdn F_n(z) = 2 \frac{b'(z',0)}{a'(z',0)}. \end{eqnarray}
Note that $\pdn F_n(z) \to -1$ as $z_n \to 0$, furthermore for $i,j \ne n$,  $\pdi F_n(z) \to 0$, $\pdj \pdi F_n(z) \to 0$ and $\pdn \pdi F_n(z) \to 0$ as $z_n \to 0$.
\\
% In order to ensure that $F$ maps points with $z_n \le 0$ to $U \cap \{ z_n \ge 0 \}$ we have to make $U$ smaller.
% Set $\tilde{U} := (U \cap \{ z_n \ge 0 \}) \cup F^{-1}(U \cap \{ z_n \le 0 \})$. Since $F$ is continuous and $F(z',0) = z'$, $\tilde{U}$ is open and for $z \in \tilde{U} \cap \R^-_n$, $F(z) \in \tilde{U}$. Replace $U$ now by $\tilde{U}$. 
Moreover the following estimate holds:
\begin{align*}
- z_n - \frac{\coeffbound}{\ellip} z^2_n \le F_n(z',z_n) \le -z_n + \frac{\coeffbound}{\ellip} z^2_n,
\end{align*}
where $\coeffbound:= \sup_{z \in \tilde{U}} |b'(z',0)|$ and $\ellip$ is the ellipticity constant mentioned after Definition \ref{DefDiffOp}.
Choosing $\varepsilon_1>0$ small enough we therefore get:
$$ F(z',z_n) \in (-\varepsilon,\varepsilon)^{(n-1)} \times (0,\varepsilon) \ \text{for} \ (z',z_n) \in (-\varepsilon,\varepsilon)^{(n-1)} \times (- \varepsilon_1,0). $$
Define $\head{U}:= (-\varepsilon,\varepsilon)^{(n-1)} \times (-\varepsilon_1,\varepsilon)$ and
% In DefDiffOp ellipticität und schranken koeff
$E_0: D(L) \to C^{2,\alpha}(\head{U})$ by:
$$
E_0u(z)=
\begin{cases}
- u (\tilde \psi \circ F(z) ) &  ,\tif \, -\varepsilon_3 < z_n <0 \\
u \big(\tilde \psi(z) \big) & ,\telse 
\end{cases}.
$$
\\
We check the smoothness of the extended function. Note that by the smoothness of $a'$ and $b'$, $F$ is a $C^{2,\alpha}$-smooth function.
\\
For points with $z_n<0$ $E_0 u$ is a composition of $u$ and the $C^{2,\alpha}$-smooth function $\tilde \psi \circ F$.
\\
Now let $z^{(0)} \in U$ with $z^{(0)}_n=0$, $y_0 := \tilde \psi(z^{(0)})$. We write $z \nearrow z^{(0)}$ for $z \to z^{(0)}$ and $z_n < 0$.
Then $\partial_{i} (u \circ \tilde \psi) (z_0)=0$, where $\partial_{i} (u \circ \tilde \psi) (z_0)$ denotes the continuous extension of the inner derivative to the boundary point $z_0$. We use the same notation for higher order derivatives below. Then $\partial_{i} (u \circ \tilde \psi \circ F)(z',z_n) \to 0$ as $z \nearrow z^{(0)}$ for $i \ne n$, so $\partial_i E_0 u$ exists in $z^{(0)}$.
By the same arguments the second partial derivatives in direction $i,j \ne n$ exist.
Moreover
\begin{align*}
\pdn (- u \circ \tilde \psi \circ F) (z) = - \pdn (u \circ \tilde \psi)(F(z)) \pdn F_n(z) \to \pdn(u \circ \tilde \psi)(z^{(0)}) \, \text{as} \, z \nearrow z^{(0)}.
\end{align*}
% Here $\pdn(u \circ \tilde \psi)(z^{(0)})$ is the continuous extension of $\pdn(u \circ \tilde \psi)$ in $\head{U} \cap \{ z_n > 0 \}$ to $\head{U} \cap \{ z_n = 0 \}$. 
By the same argument together with $\partial_{n} \partial_{i} F_n(z) \to 0$ as $z \nearrow z^{(0)}$ we get $\pdn \partial_{i} (- u \circ \tilde \psi \circ F)(z) \to \pdn \partial_{i} (u \circ \tilde \psi \circ F)(z^{(0)})$.
\\
For the second derivative in direction $n$ we have
\begin{multline*}
\pdn (- \pdn u \circ \tilde \psi \circ F)(z) = \pdn (- \pdn (u \circ \tilde \psi)(F(z)) \pdn F_n(z)) = \\
- \pdn^2 (u \circ \tilde \psi)(F(z))(\pdn F_n(z))^2 - \pdn (u \circ \tilde \psi)(F(z)) (\pdn^2 F_n(z)).  
\end{multline*}

By \eqref{ReflDirectionEq} it holds
$$ a' \pdnvtilde \pdnvtilde u (y_0) + b' \pdnvtilde u (y_0)= 0. $$
Moreover by the choice of $\tilde \psi$ it holds
$$\pdn (u \circ \tilde \psi)(z^{(0)}) = \pdnvtilde u(\tilde \psi(z^{(0)}))$$ 
and 
$$\pdn^2  (u \circ \tilde \psi)(z^{(0)}) = \pdn ((\pdnvtilde u) \circ \tilde \psi) (z^{(0)}) = \pdnvtilde \pdnvtilde u(\tilde \psi(z^{(0)})).$$
So it follows $a' \pdn^2(u \circ \tilde \psi)(z^{(0)}) + b' \pdn (u \circ \tilde \psi) (z^{(0)})= 0$.

So for $z \nearrow z^{(0)}$:
\begin{multline*}
\pdn (\pdn u \circ \tilde \psi \circ F)(z) \to - \pdn^2 (u \circ \tilde \psi)(F(z^{(0)}))(\pdn F_n(z^{(0)}))^2 - \pdn (u \circ \tilde \psi)(F(z^{(0)})) (\pdn^2 F_n(z^{(0)})) = \\ - \pdn^2 (u \circ \tilde \psi)(F(z^{(0)})) - \pdn (u \circ \tilde \psi)(F(z^{(0)})) 2 \frac{b'}{a'}(z^{(0)}) = \pdn^2 (u \circ \tilde \psi)(z^{(0)}). 
\end{multline*}

So the extensions of the one-sided first and second order derivatives from below $(z_n \le 0)$ and above $(z_n \ge 0)$ coincide.
Since the second order derivatives are Hölder continuous in both parts and continuous at the points with $z_n=0$ there are Hölder continuous in $\head{U}$.
Set $\head{V}:= \tilde \psi(\head{U})$. Then $\head{V}$ is a neighborhood of $x_0$. Define $E: C^{2,\alpha}(\head{V} \cap \Om) \to C^{2,\alpha}(\head{V})$ by:
$$ 
E u(x) = \begin{cases} E_0 u \circ \tilde \psi^{-1}(x) &, \tif \tilde \psi_n^{-1}(x) < 0 \\
u(x) &, \telse
\end{cases}.
$$
This operator fulfills the conditions \eqref{ContrLoc}, and is also bounded w.r.t to the $C^{2,\alpha}$ norm since the $C^{2,\alpha}$ norm of the extended function $Eu$ can be estimated by the $C^{2,\alpha}$ norms of $u$, $\tilde \psi$ and the coefficients $a'$, $b'$. 
\end{proof}
\begin{remark}
Note that for the $C^{2,\alpha}$-smoothness of the diffeomorphism $\tilde{\psi}$ in Theorem  \ref{LocDiffeo} we need that the transformed differential operator has $C^{2,\alpha}$-smooth second order coefficients. For the $C^{2,\alpha}$-smoothness of the $F$ in Theorem \ref{LocalEmbedding} also the first order coefficients must be $C^{2,\alpha}$-smooth. Since the first order coefficients contain second derivatives of the diffeomorphism $\psi$, $\Om$ must be assumed to be $C^{4,\alpha}$-smooth.
\end{remark}

We have therefore established the main tool for constructing the embedding operator on the whole domain.

\begin{theorem} \label{EmbeddingWholeSpace}
Let $(L,D(L))$ as in Definition \ref{DefDiffOp}, $\Om$ a bounded $C^{4,\alpha}$ smooth domain. Then there exists a linear bounded operator $E: D(L) \to C_c^{2,\alpha}(\R^n)$ with $Eu |_{\Om} = u$ and 
\begin{align}
\sup_{y \in \R^n} | Eu(y) | = \sup_{y \in \Om} \, |u|.
\label{ContrWh}
\end{align}

\end{theorem}

\begin{proof}
For $x \in \partial \Om$ let $\head{V}_x$ be the neighborhood provided by Theorem \ref{LocalEmbedding}. Then $\partial \Om \subset \bigcup_{x \in \partial \Om} \head V_x$. Since $\partial \Om$ is compact, there exist finitely many $x_i \in \partial \Om$, $1 \le i \le M$ such that $\head{V}_1,...,\head{V}_M$ cover $\partial \Om$. Denote by $(E_i,\head V_i)$ the corresponding embedding operator and neighborhood. Define $\head V_0 := \Om \setminus (\bigcup_i \head V_i)$, then by the choice of $V_i$,  $\dist(\head{V}_0,\partial \Om)>0$. Choose now a partition of unity $(\eta_i)^M_{i=0}$ such that $\eta_i$ has compact support in $\head V_i$ for $1 \le i \le M$ and $\sum_{i=0}^M \eta_i(x) = 1$ for $x \in \Omclo$. Define $E: D(L) \to C_c^{2,\alpha}(\R^n)$ by
$$ Eu := \sum_{i=0}^M \eta_i E_i(u|_{\head V_i}). $$

By the properties of $E_i$ and $\eta_i$ this defines a function in $C^{2,\alpha}_c(\R^n)$. Since $\sum_{i=0}^M \eta_i(x) = 1$ for $x \in \Omclo$, we have $Eu(x)=u(x)$ for $x \in \Om$. Since $\sum_{i=0}^M \eta_i(x) \le 1$ for $x \in \Om^c$ and $\sup_{y \in V_i} | E_i u(y) | = \sup_{y \in \head V_i \cap \Om} | u(y) |$ equality $\eqref{ContrWh}$ follows.
Furthermore the operator is also bounded (but not necessarily contractive) w.r.t the $C^{2,\alpha}$ norm. This follows from the fact that the operators $E_i$ are bounded w.r.t the $C^{2,\alpha}$ norm.
\end{proof}

{\bf Acknowledgements:} 
We dedicate this article to A. Skorokhod, I. Kovalenko and V. Korolyuk. The authors would like to thank the organizing and programme committee of the MSTAII conference, especially Yuri Kondratiev, for the opportunity to give a talk on this topic at the conference.
Furthermore we would like to thank Heinrich v. Weizsäcker and Oleg G. Smolyanov for fruitful discussions and ideas. The stay of the authors at the conference was financially supported by the DFG through project GR-1809/9-1.

\textbf{Statement:} This manuscript has not been published or submitted for publication elsewhere.


\begin{thebibliography}{}
% \bibitem[B10]{Baur} Baur, B. (2010). Core Property of Smooth Contractive Embeddable Functions for an Elliptic operator.  arXiv:1008.3520v1.
\bibitem[BGS10]{BGS} Butko, Y.A., Grothaus, M., Smolyanov, O.G. (2010). Lagrangian Feynman Formulae for Second-Order Parabolic Equations in Bounded and Unbounded Domains. Journal of Infinite Dimensional Analysis, Quantum Probability and Related Topics. Vol 13. I. 3. P. 377-392.
\bibitem[EN00]{EN}  Engel, K.-J., Nagel, R. (2000). One Parameter Semigroups for
Linear Evolution Equations. Berlin: Springer.  
\bibitem[GT77]{GilTru} Gilbarg, D., Trudinger, N.S.(1977). Elliptic PDE of second order. Springer.
\bibitem[L95]{Lunardi} Lunardi A. (1995) Analytic Semigroups and Optimal Regularity in Parabolic Problems. Birkhäuser.
\bibitem[ZJ01]{ZhJ} Zhang, G., Jiang M. (2001) Parabolic equations and
Feynman-Kac formula on general bounded domains. Sci. in China.
 V. 44. N. 3. P. 311-329.
\end{thebibliography}
\end{document}